\documentclass[10pt]{amsart}
\usepackage{amsmath, amsthm, amssymb, mathrsfs, mathtools, stmaryrd, comment, bbm, color, hyperref, adjustbox, tikz-cd}

\title{On a specialization of Toda eigenfunctions}
\author{Antoine Labelle}

\newcommand{\op}[1]{\operatorname{#1}}
\newcommand{\ul}[1]{\underline{#1}}
\newcommand{\ol}[1]{\overline{#1}}
\newcommand{\bb}[1]{\mathbb{#1}}

\newcommand{\NN}{\mathbb{N}}
\newcommand{\CC}{\mathbb{C}}
\newcommand{\ZZ}{\mathbb{Z}}
\newcommand{\QQ}{\mathbb{Q}}
\newcommand{\PP}{\mathbb{P}}
\newcommand{\KK}{\mathbb{K}}

\newcommand{\Oc}{\mathcal{O}}
\newcommand{\gf}{\mathfrak{g}}
\newcommand{\hf}{\mathfrak{h}}
\newcommand{\uf}{\mathfrak{u}}
\newcommand{\J}{\mathfrak{J}}
\newcommand{\e}{\varepsilon}

\newcommand{\qbinom}{\genfrac{[}{]}{0pt}{}}
\newcommand{\Hom}{\operatorname{Hom}}

\newcommand{\GL}{\operatorname{GL}}

\newcommand{\lam}{\lambda}
\newcommand{\pos}{{\ge 0}}
\newcommand{\one}{\mathbbm{1}}
\newcommand{\til}[1]{\widetilde{#1}}
\newcommand{\Fl}{\text{Fl}}

\newtheorem{lemma}{Lemma}
\newtheorem{theorem}{Theorem}
\newtheorem{proposition}{Proposition}
\newtheorem{definition}{Definition}
\newtheorem{example}{Example}
\newtheorem{remark}{Remark}
\newtheorem{corollary}{Corollary}
\newtheorem{conjecture}{Conjecture}

\numberwithin{lemma}{section}
\numberwithin{theorem}{section}
\numberwithin{proposition}{section}
\numberwithin{definition}{section}
\numberwithin{example}{section}
\numberwithin{remark}{section}
\numberwithin{corollary}{section}
\numberwithin{conjecture}{section}

\newcommand{\Address}{{
  \bigskip
  \footnotesize

  Antoine Labelle, \textsc{Universität Bonn}\par\nopagebreak
  \textit{E-mail address}: \texttt{labelle@math.uni-bonn.de}
}}

\begin{document}

\maketitle

\begin{abstract}
    This paper studies rational functions $\J_\alpha(q)$, which depend on a positive element $\alpha$ of the root lattice of a root system. These functions arise as Shapovalov pairings of Whittaker vectors in Verma modules of highest weight $-\rho$ for quantum groups and as Hilbert series of Zastava spaces, and are related to the Toda system. They are specializations of multivariate functions more commonly studied in the literature. We investigate the denominator of these rational functions and give an explicit combinatorial formula for the numerator in type A. We also propose a conjectural realization of the numerator as the Poincar\'e polynomial of a smooth variety in type A.
\end{abstract}



\section{Introduction}

Let $\gf$ be a split semisimple Lie algebra over $\QQ$ with Cartan subalgebra $\hf$ and let $\Phi \subseteq \hf^*$ be its root system. Fix a choice of positive roots $\Phi_+ \subseteq \Phi$ and let $\alpha_1, \ldots, \alpha_n$ be the corresponding simple roots. Let $\le$ be the partial order on $\hf^*$ defined by $\mu \le \nu$ if and only if $\nu - \mu$ is a sum of positive roots. Let $Q=\ZZ\alpha_1 \oplus \cdots \oplus \ZZ\alpha_n$ be the root lattice and $Q^{\ge 0}=\NN\alpha_1 \oplus \cdots \oplus \NN\alpha_n = \{ \alpha \in \hf^* :\alpha \ge 0\}$ be the positive part of the root lattice. Let $(\cdot, \cdot)$ be the invariant bilinear form on $\hf^*$, normalized so that short roots have length $2$. Let $d_i=\frac{(\alpha_i,\alpha_i)}{2}$ and let $a_{ij} = \frac{(\alpha_i,\alpha_j)}{d_i}$ be the entries of the Cartan matrix. 

We can associate to every $\alpha \in Q^{\ge 0}$ a rational function $\J_\alpha \in \QQ(q)$, defined by the following ``fermionic recursion":

\begin{definition} \label{def:J-function}
   Define $\J_\alpha \in \QQ(q)$, for $\alpha \in Q^{\ge 0}$, by
   \begin{equation} \label{eq:fermionic-recursion}
       \J_\alpha = \sum_{0 \le \beta \le \alpha} \frac{q^{\frac{(\beta, \beta)}{2}}}{(q)_{\alpha-\beta}} \J_\beta
   \end{equation}
   and
   \[ \J_0 = 1,\]
   where $(q)_{\alpha}=\prod_{i=1}^n \prod_{j=1}^{a_i} (1-q^{d_i j})$ for $\alpha= \sum a_i \alpha_i \in Q^{\ge 0}$.
\end{definition}

It is clear that (\ref{eq:fermionic-recursion}) together with the initial condition $\J_0=1$ uniquely determines the rational functions $\J_\alpha$ since, for $\alpha>0$, (\ref{eq:fermionic-recursion}) can be rewritten in the more obviously recursive form
\begin{equation} \label{eq:rewritten-recursion}
    \J_\alpha = \frac{1}{1-q^{\frac{(\alpha, \alpha)}{2}}} \sum_{0 \le \beta < \alpha} \frac{q^{\frac{(\beta, \beta)}{2}}}{(q)_{\alpha-\beta}} \J_\beta.
\end{equation}

These rational functions in $q$ arise naturally in different contexts. Most importantly, $\J_\alpha$ has a geometric interpretation as the Hilbert series of a Zastava space \cite{GL, BF} and a representation-theoretic interpretation as a pairing of Whittaker vectors in the Verma module of highest weight $-\rho$ for a quantum group \cite{Etingof, Sev2, FFJMM}\footnote{A connection between these two interpretations was found in type A by Braverman and Finkelberg \cite{BF2}, who defined a quantum group action on the equivariant $K$-theory of Laumon's resolution of the Zastava space, isomorphic to a Verma module.}. We will review these interpretations in Sections \ref{sec:geo-interpretation} and \ref{sec:rt-interpretation}. 

The $\J_\alpha$'s also satisfy different recursions related to the Toda lattice, which depend on a choice of orientation and a representation. In type A, with an orientation compatible with the usual ordering of the simple roots, the simplest such recursion takes the following form \cite[Proposition A.1]{FFJMM}.

\begin{proposition}[Toda recursion]
    In type $A_n$, with the standard ordering of the simple roots, and $\alpha=\sum_{i=1}^n a_i \alpha_i$, we have the recursion
    \begin{equation} \label{eq:Toda-recursion}
        \left(\sum_{i=0}^n (q^{a_{i+1}-a_i} -1)\right) \J_\alpha =  \sum_{i=1}^n q^{a_{i+1}-a_i} \J_{\alpha-\alpha_i},
    \end{equation}
    with the conventions that $a_i=0$ if $i\not\in\{1, \ldots, n\}$ and $\J_\alpha=0$ if $\alpha \not\in Q^\pos$.
\end{proposition}

\begin{remark} \label{rmk:multivariate}
    In the literature, the functions $\J_\alpha$ usually appear with $n$ extra variables $z_1, \ldots, z_n$; they can be defined by the recursion
    \[\J_\alpha(q, z_1, \ldots, z_n) = \sum_{0 \le \beta \le \alpha} \frac{z^\beta q^{\frac{(\beta, \beta)}{2}}}{(q)_{\alpha-\beta}} \J_\beta(q, z_1, \ldots, z_n),\]
    where $z^\beta = \prod_{i=1}^n z_i^{b_i}$ for $\beta = \sum b_i \beta_i \in Q^{\ge 0}$. The $\J_\alpha$ defined in Definition \ref{def:J-function} is the specialization at $z_1=\cdots=z_n=1$ of this multivariate version. 
    Under the geometric interpretation of Section \ref{sec:geo-interpretation}, this multivariate enhancement encodes the character of a $T\times \CC^\times$ action on the Zastava space rather than just a $\CC^\times$ action. Under the representation-theoretic interpretation of Section \ref{sec:rt-interpretation}, it corresponds to considering Whittaker vectors in Verma modules of arbitrary highest weight (or in a universal Verma module). 
\end{remark}

While the particular specialization that we consider here has not, to the author's knowledge, been studied by itself before, it has some nice properties that do not generalize to the multivariate enhancement of Remark \ref{rmk:multivariate}. In particular, we will prove in Section \ref{sec:denominator} that $\J_\alpha$ has a very simple denominator (which is not explained by the recursion (\ref{eq:fermionic-recursion})):

\begin{theorem} \label{thm:denominator}
    For every $\alpha\in Q^\pos$, we have
    \[ \J_\alpha \in \frac{1}{(q)_\alpha^2}\ZZ[q].\]
\end{theorem}

Moreover, the numerator $(q)_\alpha^2\J_\alpha$ is a palindromic polynomial (see Remark \ref{rmk:palindromic}), which we expect to have positive coefficients (see Conjecture \ref{conj:pos-coeffs}) and, in simply laced type, to be unimodal (see Conjecture \ref{conj:unimodal}).

In addition to Theorem \ref{thm:denominator}, the other main contribution of this paper is the following explicit combinatorial formula for this numerator in type A, which will be proven in Section \ref{sec:formula}. 

\begin{theorem} \label{thm:combi-formula}
    In type $A_n$, with the standard ordering of the simple roots, and $\alpha=\sum_{i=1}^n a_i \alpha_i$, we have
    \begin{equation} \label{eq:main-formula}
    \J_\alpha(q) = \frac{1}{(q)_\alpha^2}\sum_{\underline{m}} q^{D(\underline{m})}\prod_{k=1}^{n-1} \prod_{i=1}^k \qbinom{m_{k+1,i}}{m_{k,i}}_q \qbinom{m_{k+1,i+1}}{m_{k,i}}_q,
    \end{equation}
    where the sum is over all triangular arrays $\underline{m}=(m_{k,i})_{1\le i\le  k \le n}$ of nonnegative integers satisfying $m_{n,i}=a_i$ and $m_{k,i} \le m_{k+1,i},m_{k+1,i+1}$ for every $1\le i \le k \le n-1$,
    \[ D(\underline{m})= \sum_{k=1}^{n-1} \left(\sum_{i=1}^k m_{k,i}^2 - \sum_{i=1}^{k-1}m_{k,i}m_{k,i+1}\right),\]
    and
    \[ \qbinom{a}{b}_q = \frac{\prod_{j=1}^{a} (1-q^j)}{\prod_{j=1}^{b} (1-q^j) \prod_{j=1}^{a-b} (1-q^j)} \]
    are the Gaussian binomial coefficients.
\end{theorem}

Theorem \ref{thm:combi-formula} can be seen as a $q$-analog of (special cases of) formulas of Ishii and Stade \cite{IS} and of Batyrev et al. \cite{BCKvS}. See Remark \ref{rmk:K-limit} for more details.

\begin{example} \label{ex:A2}
    In type $A_2$, if $\alpha=a_1\alpha_1 + a_2\alpha_2$, the formula (\ref{eq:main-formula}) gives
    \[ \J_\alpha = \frac{1}{(q)_\alpha^2}\sum_{i=0}^{\min(a_1,a_2)} q^{i^2}\qbinom{a_1}{i}_q \qbinom{a_2}{i}_q, \]
    which simplifies to 
    \[ \J_\alpha = \frac{1}{(q)_\alpha^2}\qbinom{a_1+a_2}{a_1}_q.\]
\end{example}

\begin{example} \label{ex:A3}
    In type $A_3$, if $\alpha = a_1\alpha_1 + a_2\alpha_2 + a_3\alpha_3$, the formula (\ref{eq:main-formula}) gives
    \[ \J_\alpha(q) = \frac{1}{(q)_\alpha^2}\sum_{i}\sum_{j}\sum_{k} q^{i^2+j^2-ij +k^2} \qbinom{a_1}{i}_q \qbinom{a_2}{i}_q \qbinom{a_2}{j}_q \qbinom{a_3}{j}_q \qbinom{i}{k}_q \qbinom{j}{k}_q, \]
    where the sums are over $0\le i \le a_1, a_2$, $0\le j\le a_2,a_3$ and $0\le k \le i,j$.
    The sum over $k$ simplifies as in Example \ref{ex:A2}, so we can also write this as
    \[ \J_\alpha(q) = \frac{1}{(q)_\alpha^2} \sum_{i=0}^{\min(a_1,a_2)} \sum_{j=0}^{\min(a_2,a_3)} q^{i^2+j^2-ij} \qbinom{a_1}{i}_q \qbinom{a_2}{i}_q \qbinom{a_2}{j}_q \qbinom{a_3}{j}_q \qbinom{i+j}{i}_q.\]
    As an interesting special case, if we take $\alpha=m(\alpha_1+\alpha_2+\alpha_3)$ for $m\in \NN$ and specialize at $q=1$, we obtain
    \[ (q)_\alpha^2\J_\alpha|_{q=1} = \sum_{i=0}^{m} \sum_{j=0}^{m}\binom{m}{i}^2\binom{m}{j}^2 \binom{i+j}{i}.\]
    By \cite{MO}, this is equal to the $m^\text{th}$ Ap\'ery number \cite{OEIS}. The polynomials $(q)_\alpha^2\J_\alpha$ for this particular sequence of $\alpha$'s can therefore be seen as $q$-analogs of Apéry numbers.
\end{example}

\begin{example}  \label{ex:narayana}
    For $\alpha=\alpha_1+\ldots+\alpha_n$ the highest root in type $A_n$, the formula  (\ref{eq:main-formula}) yields a Narayana polynomial as the numerator of $\J_\alpha$. See Remark \ref{rmk:narayana} for more details.
\end{example}

Finally, in Section \ref{sec:top-interpretation}, we propose a third interpretation for the function $\J_\alpha$, or more specifically its numerator. In that section, we define (in type A) a smooth projective variety and we conjecture that the Poincaré polynomial of this variety is given by $(q)_\alpha^2\J_\alpha$. We give some evidence for this conjecture, notably by proving the $n=3$ case.

\begin{remark} \label{rmk:K-limit}
    Let $\alpha=\sum_i a_i \alpha_i$. It's easy to see from (\ref{eq:rewritten-recursion}) that $\J_\alpha$ has a pole of order at most $2|\alpha|$ at $q=1$, where $|\alpha|=\sum_i a_i$. We can therefore define
    \[ K_\alpha = \left. \left( \prod_{i=1}^n (1-q^{d_i})^{2a_i} \cdot \J_\alpha \right)\right|_{q=1} \in \QQ.\]
    Multiplying both sides of (\ref{eq:rewritten-recursion}) by $\prod_{i=1}^n (1-q^{d_i})^{2a_i}$ and taking the limit as $q\to 1$, we obtain the following recursion for $K_\alpha$\footnote{Note that the $q\to 1$ limit of the Toda recursion (\ref{eq:Toda-recursion}) yields the same recursion as the $q\to 1$ limit of the fermionic recursion (\ref{eq:rewritten-recursion}).}: 
    \begin{equation} \label{eq:K-recursion}
    \frac{(\alpha, \alpha)}{2} \cdot K_\alpha = \sum_{i=1}^n d_i K_{\alpha-\alpha_i}
    \end{equation}
    (as all the terms corresponding to $\beta<\alpha$ such that $\alpha-\beta$ is not a simple root vanish in the $q\to 1$ limit).
    
    Under the geometric interpretation of Section \ref{sec:geo-interpretation}, this $q\to 1$ specialization controls the leading term of the Hilbert function of the Zastava space, i.e.
     \[ \dim \CC[Z^\alpha]_n = K_\alpha \frac{n^{2|\alpha|-1}}{(2|\alpha|-1)!} + O(n^{2|\alpha|-2}).\]
    Under the representation-theoretic interpretation of Section \ref{sec:rt-interpretation}, it corresponds to taking Whittaker vectors for $\gf$ itself rather than its quantum group. Under the conjectural interpretation of Section \ref{sec:top-interpretation}, $K_\alpha$ is related to the Euler characteristic of the variety $X_\alpha$ defined in that section. Specifically, Conjecture \ref{conj:poincare-poly} would imply that \[ K_\alpha = \frac{\chi(X_\alpha)}{\left(\prod_{i=1}^na_i!\right)^2}.\] 
    This $q\to 1$ specialization of $\J_\alpha$ is already quite nontrivial and interesting. For example, it follows from Theorem \ref{thm:denominator} that the denominator of $K_\alpha$ divides $\left(\prod_{i=1}^na_i!\right)^2$, which is not at all clear from the recursion (\ref{eq:K-recursion}). Moreover, the formula (\ref{eq:main-formula}) specializes to the following formula for $K_\alpha$:
    \begin{equation} \label{eq:K-combi-formula}
    K_\alpha = \frac{1}{\left(\prod_{i=1}^na_i!\right)^2}\sum_{\underline{m}} \prod_{k=1}^{n-1}\prod_{i=1}^k \binom{m_{k+1,i}}{m_{k,i}} \binom{m_{k+1,i+1}}{m_{k,i}},
    \end{equation}
    with the sum indexed by the same set as in Theorem \ref{thm:combi-formula}. The formula (\ref{eq:K-combi-formula}) has already appeared before in the literature. In the context of hypergeometric series of partial flag varieties, it appears in \cite[Theorem 5.1.6]{BCKvS} (for the particular case of \emph{complete} flag varieties), and in the context of Whittaker functions it is a special case of a formula of Ishii and Stade \cite[Theorem 18]{IS}\footnote{Specifically, the $a=0$ case.}. Our Theorem \ref{thm:combi-formula} can therefore be seen as a $q$-analog of the formulas cited just above. See also \cite{OCon}, where a probabilistic interpretation of (\ref{eq:K-combi-formula}) is studied.

\end{remark}

\section{Geometric interpretation} \label{sec:geo-interpretation}

In this section, we assume for simplicity that $\gf$ is simply laced (i.e. $d_i=1$ for all $i$). Fix $\alpha=\sum a_i\alpha_i\in Q^\pos$.

Let $Z^\alpha$ be the Zastava space of degree $\alpha$ based quasimaps from $\PP^1$ to the flag variety of $\gf$ (whose definition can be found for example in \cite{FM}). The Zastava space $Z^\alpha$ is an affine variety of dimension $2|\alpha|$. It comes equipped with a $\CC^\times$ action given by loop rotation, which induces an $\NN$-grading on its coordinate ring $\CC[Z^\alpha]$. It also comes with a $\CC^\times$-equivariant flat morphism $Z^\alpha \to \bb{A}^\alpha$, where $\bb{A}^\alpha = \prod_{i=1}^n (\bb{A}^{a_i}/S_{a_i})$, called the \emph{factorization} map.
The following interpretation of $\J_\alpha$ is due to Givental and Lee \cite{GL} in type A and to Braverman and Finkelberg in this level of generality \cite{BF}. We give here an alternative simple proof based on the interpretation of $Z^\alpha$ as the Coulomb branch of a quiver gauge theory due to Braverman, Finkelberg and Nakajima \cite{BFN2} and the monopole formula for Coulomb branches.

\begin{theorem} \label{thm:geo-interpretation}
    The rational function $\J_\alpha$ is equal to the Hilbert series of the graded ring $\CC[Z_\alpha]$, i.e.
    \[ \J_\alpha(q) = \sum_{d\ge 0}\dim(\CC[Z^\alpha]_d) q^d,\]
    where $\CC[Z^\alpha]_d$ denotes the degree $d$ part of $\CC[Z^\alpha]$.
\end{theorem}

\begin{proof}
    Let $G=\prod_i \GL_{a_i}(\CC)$ and consider the $G$-representation \[N=\oplus_{(i,j)\in E}\Hom(\CC^{a_i}, \CC^{a_j}),\] where $E$ is the set of edges of the Dynkin diagram of $\Phi$. Let $\pi: \mathcal{R} \to \op{Gr}_G$ be the BFN space of triples associated to the pair $(G,N)$ as defined in \cite{BFN} and let $\mathcal{R}^+ = \pi^{-1}(\op{Gr}_G^+)$, where $\op{Gr}_G^+$ is the positive part of the affine Grassmannian of $G$. For $\Oc=\CC[\![z]\!]$, the $G(\Oc)$-orbits in $\op{Gr}_G$ are indexed by tuples $\lam=(\lam^{(1)}; \ldots; \lambda^{(n)})$, where $\lam^{(i)}=(\lam^{(i)}_1, \ldots, \lam^{(i)}_{a_i})$ is a weakly decreasing sequence of integers of length $a_i$. Then $\op{Gr}_G^+ \subseteq \op{Gr}_G$ is the union of $G(\Oc)$-orbits indexed by those $\lam$ that are \emph{multipartitions}, i.e. such that $\lam_j^{(i)}\ge 0$ for all $i,j$.
    
    By \cite[Corollary 3.4]{BFN2}, $\CC[Z_\alpha]$ is isomorphic to the Coulomb branch algebra $H_*^{G(\Oc)}(\mathcal{R}^+)$. Moreover, while the loop rotation grading $\CC[Z_\alpha]$ does not agree with the homological grading on $H_*^{G(\Oc)}(\mathcal{R}^+)$, they are closely related as explained in \cite[Remark 3.3]{BFN2}. Therefore, the monopole formula of \cite[Section 2(iii)]{BFN} (appropriately modified to take into account the correction of \cite[Remark 3.3]{BFN2} and restricted to positive $G(\Oc)$-orbits) yields
    \begin{equation} \label{eq:monopole}
        \sum_{d\ge 0}\dim(\CC[Z^\alpha]_d) q^d = \sum_{\lam \in \op{MPar}_\alpha} q^{d_\lam} P_\lam(q),
    \end{equation} 
    where $\op{MPar}_\alpha$ is the set of all multipartitions $\lam=(\lam^{(1)}, \ldots, \lam^{(r)})$ with $\lam^{(i)}= (\lam^{(1)}_1, \ldots, \lam^{(1)}_{a_i})$,
    \[P_\lam(q) = \prod_{i=1}^r \prod_{j\ge 0} \prod_{k=1}^{m_j(\lam^{(i)})} \frac{1}{1-q^k},\]
    for $m_j(\lam^{(i)})$ the multiplicity of $j$ in the partition $\lam^{(i)}$, and
    \[ d_\lam = \sum_{i=1}^r \sum_{j=1}^{a_i} (2j-1) \lam^{(i)}_j - \sum_{(i_1, i_2)\in E} \sum_{j_1=1}^{a_{i_1}} \sum_{j_2=1}^{a_{i_2}} \min(\lam^{(i_1)}_{j_1}, \lam^{(i_2)}_{j_2}).\]
    If we denote by $F_\alpha(q)$ the right-hand side of (\ref{eq:monopole}), we therefore just need to check that $F_
    \alpha$ satisfies the fermionic recursion defining $\J_\alpha$ (since clearly $F_0=1$). For this, we split the sum over $\lam$ into different sums depending on which $\lam_j^{(i)}$ vanish. For $\beta = \sum b_i\alpha_i\le \alpha$, let $\op{MPar}_{\alpha, \beta}$ be the set of multipartitions $\lam \in \op{MPar}_{\alpha}$ such that $\lam^{(i)}_j =0$ if and only if $j>b_i$. Then $\op{MPar}_{\alpha} =\bigsqcup_{0\le \beta \le \alpha} \op{MPar}_{\alpha, \beta}$ and there is a bijection $\sigma_{\alpha, \beta} : \op{MPar}_{\beta}\overset{\sim}{\to} \op{MPar}_{\alpha, \beta}$ sending a multipartition $\mu =(\mu^{(i)})_i$ to a multipartition $\lam =(\lam^{(i)})_i$ where
    \[ \lam^{(i)} = (\mu^{(i)}_1 + 1, \ldots, \mu^{(i)}_{b_i}+1, 0 \ldots, 0).\]
    It's easy to see that 
    \[ P_{\sigma_{\alpha, \beta}(\mu)}(q) = \frac{1}{(q)_{\alpha-\beta}} P_\mu(q)\]
    and 
    \[ d_{\sigma_{\alpha, \beta}(\mu)} = d_\mu + \sum_{i=1}^nb_i^2 - \sum_{(i_1,i_2)\in E} b_i b_j = d_\mu + \frac{(\beta, \beta)}{2}. \]
    Hence, we have
    \begin{align*}
        F_\alpha(q) &= \sum_{\lam \in \op{MPar}_\alpha} q^{d_\lam} P_\lam(q) \\
        &= \sum_{0\le \beta\le \alpha}\sum_{\lam \in \op{MPar}_{\alpha, \beta}} q^{d_\lam} P_\lam(q) \\
        &= \sum_{0\le \beta\le \alpha}\sum_{\mu \in \op{MPar}_{\beta}} q^{d_\mu+ \frac{(\beta,\beta)}{2}} \frac{P_\mu(q)}{(q)_{\alpha-\beta}} \\
        &= \sum_{0\le \beta\le \alpha}\frac{q^{\frac{(\beta,\beta)}{2}}}{(q)_{\alpha-\beta}} F_\beta(q).
    \end{align*} 
\end{proof}

\begin{remark} \label{rmk:central-fiber}
    Note that, since $\bb{A}^\alpha$ has Hilbert series $\frac{1}{(q)_\alpha}$, it follows from Theorem \ref{thm:geo-interpretation} that $(q)_\alpha \J_\alpha$ is the Hilbert series of the central fiber $Z^\alpha_0$ of the flat morphism $Z^\alpha\to \bb{A}^\alpha$. One can therefore recover $\J_\alpha$ from only this central fiber rather than the whole Zastava space. This central fiber is isomorphic to the intersection $S^-_0 \cap \ol{S_\alpha}$ of two opposite semi-infinite orbits in the affine Grassmannian of the adjoint group with Lie algebra $\gf$ \cite[6.4.2]{FM}.
\end{remark}

\begin{remark}
    The geometric interpretation presented in this section was extended to the non-simply laced case by Braverman and Finkelberg in \cite{BF3}. For this, the Zastava spaces need to be replaced by twisted Zastava spaces. 
\end{remark}

\section{Representation-theoretic interpretation} \label{sec:rt-interpretation}

In this section, we recall how $\J_\alpha$ arises from Whittaker functions for Verma modules over quantum groups, following closely the presentation in \cite{FFJMM}. These were first studied independently by Etingof \cite{Etingof} and Sevostyanov \cite{Sev2}, who both observed that the Whittaker functions satisfy a Toda recursion. The fermionic recursion (\ref{eq:fermionic-recursion}) in this context was proved in \cite{FFJMM}. See also \cite{DKT} for closely related formulas. All these references work with Verma modules of arbitrary highest weight (or equivalently, a universal Verma module), which leads to an interpretation of the multivariate functions from Remark \ref{rmk:multivariate}. For the specialization studied in this paper, the relevant Verma module is that of highest weight $-\rho$, where $\rho$ is the Weyl vector defined by $(\rho, \alpha_i)=d_i$ for all $i$.

Let $\mathbf{U}_v$ be the Drinfeld--Jimbo quantum group with parameter $v$ associated to the Lie algebra $\gf$, over the field $\KK=\QQ(v^{\frac{1}{N}}\mid N\in \NN^*)$. It is defined as the associative $\KK$-algebra generated by $E_i, F_i$ for $1\le i \le n$ and $K_\mu$ for $\mu \in \hf^*$ with relations 

\begin{gather}
    K_\mu K_\nu = K_{\mu+\nu}, \quad K_0=1, \\
    K_\mu E_i K_\mu^{-1} = v^{(\mu, \alpha_i)} E_i, \quad K_\mu F_i K_\mu^{-1} = v^{-(\mu, \alpha_i)} F_i, \\
    [E_i, F_j] = \delta_{ij} \frac{K_i-K_i^{-1}}{v-v^{-1}},
    \\
    \sum_{k=0}^{r} (-1)^k E_i^{(r-k)} E_j E_i^{(k)} = 0 \quad \text{for $i\ne j, r=1-a_{ij}$}, \label{eq:serre-rel}\\
    \sum_{k=0}^{r} (-1)^k F_i^{(r-k)} F_j F_i^{(k)} = 0 \quad \text{for $i\ne j, r=1-a_{ij}$},
\end{gather}
where 
\[ E_i^{(k)} = \frac{E_i^k}{[k]_{v^{d_i}}!}, \quad F_i^{(k)} = \frac{F_i^k}{[k]_{v^{d_i}}!}\]
and
\[ [k]_u!= \displaystyle\prod_{j=1}^k \left(\frac{u^{j}-u^{-j}}{u-u^{-1}}\right).
\footnote{The convention for $[k]_u!$ that we use here and in Section \ref{sec:denominator} is different to the convention for Gaussian binomial coefficients used in the introduction and in Sections \ref{sec:formula} and \ref{sec:top-interpretation}, as this is the standard convention in the context of quantum group.}\]

Similarly, we let $\mathbf{U}_{v^{-1}}$ be the quantum group with parameter $v^{-1}$, whose generators are denoted by $\overline{E}_i, \overline{F}_i, \overline{K}_\mu$. There is a $\KK$-linear anti-isomorphism $\omega : \mathbf{U}_v \to \mathbf{U}_{v^{-1}}$ defined by
\[\omega(E_i)=\overline{F}_i,\quad \omega(F_i)=\overline{E}_i,\quad \omega(K_\mu)=\overline{K}_{-\mu}.\]

Given $\lambda\in \hf^*$, the Verma module $\mathbf{V}^\lambda$ is the left $\mathbf{U}_v$-module generated by a single element $\one_\lambda$ with relations $E_i \one_\lam =0$ and $K_\mu \one_\lam = v^{(\mu, \lam)} \one_\lam$. We have a weight decomposition \[\mathbf{V}^\lam=\bigoplus_{\alpha\in Q^\pos}(\mathbf{V}^\lam)_\alpha,\]
where $(\mathbf{V}^\lam)_\alpha=\{x\in \mathbf{V}^\lam : \forall \mu \in \hf^*, K_\mu x = v^{(\mu, \lam-\alpha)} x\}$ is the weight space of weight $\lam-\alpha$. Let
\[\widehat{\mathbf{V}^\lam}=\prod_{\alpha\in Q^\pos}(\mathbf{V}^\lam)_\alpha\]
be the completion of $\mathbf{V}^\lam$ with respect to the weight grading. Similarly, $\ol{\mathbf{V}}^\lam$ is the Verma module for $\mathbf{U}_{v^{-1}}$ of highest weight $\lam$, which is generated by $\ol{\one}_\lam$ with relations $\ol{E}_i\ol{\one}_\lam =0$ and $\ol{K}_\mu\ol{\one}_\lam =v^{-(\mu, \lam)}\ol{\one}_\lam$ and has a similar weight decomposition. There is a unique $\KK$-bilinear form $(\cdot, \cdot):\mathbf{V}^\lam \times \ol{\mathbf{V}}^\lam \to \KK$, satisfying $(\one_\lam, \ol{\one}_\lam) = 1$ and $(rx, y)=(x,\omega(r)y)$ for all $r\in \mathbf{U}_v$, which is called the Shapovalov form.

Choose an orientation of the Dynkin diagram associated to the root system, which we encode by matrix $(\e_{ij})_{i=1}^n$, where
\[ \e_{ij} =\begin{cases}
    1 &\text{if there is an edge oriented from $i$ to $j$}. \\
    -1 &\text{if there is an edge oriented from $j$ to $i$}. \\
    0 &\text{if there is no edge from $i$ to $j$}.
    \end{cases} \]
Pick also a tuple of weights $(\nu_i)_{i=1}^n \in (\hf^*)^n$ satisfying
\[ (\nu_i, \alpha_j) - (\nu_j, \alpha_i) = \e_{ij} (\alpha_i, \alpha_j).\footnote{While not unique, such a tuple always exists. For example, after having fixed an ordering of the simple roots, we can take $\nu_i = \sum_{k=0}^{i-1} \omega_k  \e_{ik} a_{ki}$, where $\omega_k$ is the fundamental weight, defined by $(\omega_k, \alpha_j) = \delta_{kj}d_k$.}\]
There exists a unique vector $\theta^\lam = \sum_{\alpha\in Q^\pos} \theta^\lam_\alpha \in \widehat{\mathbf{V}^\lam}$, with $\theta^\lam_\alpha\in \mathbf{V}^\lam_\alpha$, satisfying $\theta^\lam_0 = \one^\lam$ and
\[ E_i K_{\nu_i}\theta^\lam = \theta^\lam\]
for every $i$, called the \emph{Whittaker vector} (relative to the choice of orientation and tuple $(\nu_i)$)\footnote{In the classical setting of Verma modules for semisimple Lie algebras, Whittaker vectors are simply eigenvectors for the Chevalley generators $e_i$. Such a definition cannot carry verbatim to the setting of quantum groups, since it follows from the quantum Serre relation (\ref{eq:serre-rel}) the algebra generated by the $E_i$ admits no nondegenerate character (see \cite{Sev1}). This is why we need to introduce the $K_{\nu_i}$ term.}. Similarly, let $\bar{\theta}^\lam \in \widehat{\ol{\mathbf{V}}^\lam}$ be the unique vector satisfying $\bar{\theta}^\lam_0 = \ol{\one}^\lam$ and
\[ \ol{E}_i \ol{K}_{\nu_i}\bar{\theta}^\lam = \bar{\theta}^\lam\]
for every $i$. 

Up to a simple factor, $\J_\alpha$ can be recovered as the Shapovalov pairing of $\theta^\lam_\alpha$ with $\bar{\theta}^\lam_\alpha$ for $\lam=-\rho$ \cite[Theorem 3.1]{FFJMM}:

\begin{theorem} \label{thm:rt-interpretation}
    For every $\alpha \in Q^\pos$,
    \begin{equation} \label{eq:rt-interpretation}
        \J_\alpha |_{q=v^2} = v^{-\frac{(\alpha, \alpha)}{2}-(\rho,\alpha)} \prod_i \frac{1}{\left((1-v^{2d_i})(1-v^{-2d_i})\right)^{a_i}} (\theta^{-\rho}_\alpha,\bar{\theta}^{-\rho}_\alpha).
    \end{equation}
\end{theorem}

\begin{remark}
    It is not hard to see that the right-hand side of (\ref{eq:rt-interpretation}) is independent of the choice of $(\nu_i)$, as there are simple rules for how the Whittaker vectors transform under a modification of the $(\nu_i)$ \cite{FFJMM}. A more surprising consequence of Theorem \ref{thm:rt-interpretation} is that the right-hand side is also independent of the choice of orientation. There does not seem to be a known direct proof of this, other than showing that the right-hand side satisfies the same recursion for any orientation. 
\end{remark}

From this interpretation of $\J_\alpha$, we can easily see that $\J_\alpha$ satisfies a symmetry property with respect to $q\to q^{-1}$ (which is not obvious from the defining recursion (\ref{eq:fermionic-recursion})).

\begin{corollary} \label{cor:symmetry}
    For every $\alpha \in Q^\pos$,
    \[ \J_\alpha(q^{-1}) = q^{\frac{(\alpha, \alpha)}{2}+(\rho,\alpha)} \J_\alpha(q).\]
\end{corollary}

\begin{proof}
    Since $v$ and $v^{-1}$ play a symmetric role, it's clear that the term \[\prod_i \frac{1}{\left((1-v^{2d_i})(1-v^{-2d_i})\right)^{a_i}} (\theta^{-\rho}_\alpha,\bar{\theta}^{-\rho}_\alpha)\] is invariant under $v \to v^{-1}$. The result then follows immediately from Theorem \ref{thm:rt-interpretation}.
\end{proof}

\section{Denominator} \label{sec:denominator}

In this section, we use the representation-theoretic interpretation from Section \ref{sec:rt-interpretation} to prove Theorem \ref{thm:denominator}. This is really a special property of the specialization $\J_\alpha$ studied in this paper and does not extend to the multivariate enhancement mentioned in Remark \ref{rmk:multivariate}. Indeed, it is crucial that we work specifically with the Verma module of highest weight $-\rho$. The key input is that, for this particular Verma module,  the Shapovalov form is a perfect pairing on Lusztig's integral form.

Let $A=\ZZ[v^{\pm \frac{1}{N}}\mid N\in \NN^*]$ and let $\mathcal{U}_v$ be Lusztig's integral form of $\mathbf{U}_v$, which is the $A$-subalgebra of $\mathbf{U}_v$ generated by the $K_\mu$'s and the divided powers $E_i^{(k)}$, $F_i^{(k)}$ \cite{Lusztig}. Let $\mathcal{V}^\lam=\mathcal{U}_v\cdot \one^\lam$ be the corresponding integral form of the Verma module $\mathbf{V}^\lam$ and $\mathcal{V}^\lam_\alpha=\mathcal{V}^\lam\cap \mathbf{V}^\lam_\alpha$. Given a choice of reduced word for the longest word in the Weyl group, one can construct elements $F_{\beta} \in \mathcal{U}_v$ for every $\beta\in \Phi_+$ and an ordering $\beta_1, \ldots, \beta_N$ of the positive roots, so that that the divided power monomials
\[ F^{(\vec{k})} \one^\lam := F_{\beta_1}^{(k_{\beta_1})} \cdots F_{\beta_N}^{(k_{\beta_N})} \one^\lam \]
form an $A$-basis of $\mathcal{V}^\lam_\alpha$ as $\vec{k}=(k_\beta)_{\beta\in \Phi_+}$ runs over Kostant partitions of $\alpha$ (i.e. $\vec{k}=(k_\beta)_{\beta\in \Phi_+}\in \NN^{\Phi_+}$ and $\sum_{\beta\in \Phi_+} k_\beta \beta = \alpha$) \cite[Theorem 6.7]{Lusztig}. Here \[F_\beta^{(k)}=\frac{F_\beta^k}{[k]_{v^{d_\beta}}!},\] for $d_\beta=\frac{(\beta,\beta)}{2}$. Define $\mathcal{U}_{v^{-1}}$, $\ol{\mathcal{V}}^\lam$, $\ol{\mathcal{V}}^\lam_\alpha$ and $\ol{F}^{(\vec{k})} \ol{\one}^\lam$ analogously.

Note that the Shapovalov form restricted to $\mathcal{V}^\lam \times \ol{\mathcal{V}}^\lam$ takes values in $A$, since $(x \one^\lam, y \ol{\one}^\lam)=(\one^\lam,\omega(x)y\ol{\one}^\lam)$ for every $x\in \mathcal{U}_v$, $y\in \mathcal{U}_{v^{-1}}$ and the right-hand side is the coefficient of $\ol{\one}$ in $\omega(x)y\ol{\one}\in \ol{\mathcal{V}}^\lam$, which is an element of $A$. In the case $\lam=-\rho$, it turns out that we have a perfect pairing over $A$.

\begin{proposition} \label{prop:perfect-pairing}
    For every $\alpha \in Q^\pos$, the Shapovalov form $\mathcal{V}^{-\rho}_\alpha \times \ol{\mathcal{V}}^{-\rho}_\alpha \to A$ is a perfect pairing.
\end{proposition}

\begin{proof}
    This is equivalent to the determinant of the Shapovalov pairing with respect to the bases $(F^{(\vec{k})} \one^{-\rho})_{\vec{k}}$ and $(\ol{F}^{(\vec{k})} \ol{\one}^{-\rho})_{\vec{k}}$ being a unit in $A$. This determinant was calculated by De Concini and Kac for any $\lam$ \cite[Proposition 1.9]{DK}, generalizing a formula of Shapovalov \cite{Sha} for Verma modules over Lie algebras. De Concini and Kac work with the basis of monomials $F_{\beta_1}^{k_{\beta_1}} \cdots F_{\beta_N}^{k_{\beta_1}} \one^\lam$, for which their formula yields (when specialized to $\lam=-\rho$) a determinant of
    \[ \pm \prod_{\beta\in \Phi_+} \prod_{m\in \NN} \left(\frac{v^{md_\beta}-v^{-md_\beta}}{v^{d_\beta}-v^{-d_\beta}}\right)^{2|\mathrm{KP}(\alpha-m\beta)|},\]
    where $\op{KP}(\alpha)$ is the set of Kostant partitions of $\alpha$. This can be rewritten as
    \begin{align*}
        &\pm \prod_{\beta\in \Phi_+} \prod_{\vec{k} \in \op{KP}(\alpha)} \prod_{m=1}^{k_\beta} \left(\frac{v^{md_\beta}-v^{-md_\beta}}{v^{d_\beta}-v^{-d_\beta}}\right)^2 \\
        = &\pm \left(\prod_{\vec{k} \in \op{KP}(\alpha)} \prod_{\beta\in \Phi_+} [k_\beta]_{v_{\beta}}!\right)^2
    \end{align*}
    since there is a bijection between pairs of a natural number $m$ and a Kostant partition of $\alpha-m\beta$ and pairs of a Kostant partition $\vec{k}$ of $\alpha$ and a natural number $m \le k_\beta$ (given by increasing $k_\beta$ by $m$). But this is (up to a sign) exactly the square of the determinant of the change of basis matrix between De Concini-Kac's basis and the divided power basis, so the determinant of the Shapovalov form in the divided power basis is $\pm 1$.
\end{proof}

With this proposition established, we are now ready to prove Theorem \ref{thm:denominator}.

\begin{proof}[Proof of Theorem \ref{thm:denominator}]
    We will prove by induction on $\alpha$ that
    \[ \theta_\alpha^{-\rho} \in \frac{1}{[\alpha]_v!}\mathcal{V}_\alpha^{-\rho},\]
    where 
    \[[\alpha]_v!= \prod_{i=1}^n [a_i]_{v^{d_i}}! =\frac{v^{-\sum_i \binom{a_i}{2}}}{\prod_{i=1}^n(1-v^{2d_i})^{a_i}} \left. (q)_\alpha\right|_{q=v^2}\]
    if $\alpha=\sum_{i=1}^na_i\alpha_i$.
    The base case $\alpha=0$ is clear since $\theta_0^{-\rho}=\one^{-\rho}$. If $\alpha>0$, then $\ol{\mathcal{V}}_\alpha^{-\rho}$ is spanned (over $A$) by vectors of the form $\ol{F}_i^{(k)}x$ for some $k>0$, $1\le i \le n$ and $x \in \ol{\mathcal{V}}_{\alpha-k\alpha_i}^{-\rho}$. We have
    \begin{align*}
        (\theta^{-\rho}_\alpha,\ol{F}_i^{(k)}x) &= (E_i^{(k)}\theta^{-\rho}_\alpha,x) \\
        &= \frac{1}{[k]_{v_i}!}(E_i^{k}\theta^{-\rho}_\alpha,x) \\
        &= \frac{v^r}{[k]_{v_i}!}(\theta^{-\rho}_{\alpha-k\alpha_i},x) \\
       &\in \frac{1}{[k]_{v_i}! [\alpha-k\alpha_i]_v!} A \subseteq \frac{1}{[\alpha]_v!}A
    \end{align*}
    for some $r\in \QQ$ (which depends on the data $(\nu_i)$), where we used the induction hypothesis and the fact that $[k]_{v_i}![a_i-k]_{v_i}!$ divides $[a_i]_{v_i}!$ in $A$ on the last line. This shows that
    \[ (\theta^{-\rho}_\alpha,\ol{\mathcal{V}}_\alpha^{-\rho}) \subseteq \frac{1}{[\alpha]_v!}A, \]
    so $\theta^{-\rho}_\alpha \in \frac{1}{[\alpha]_v!} \mathcal{V}_\alpha^{-\rho}$
    by Proposition \ref{prop:perfect-pairing}. Similarly, we have $\bar{\theta}^{-\rho}_\alpha \in \frac{1}{[\alpha]_v!} \ol{\mathcal{V}}_\alpha^{-\rho}$, so
    \[ (\theta^{-\rho}_\alpha,\bar{\theta}^{-\rho}_\alpha) \in \frac{1}{\left([\alpha]_v!\right)^2} A.\]
    Setting $q=v^2$ and applying Theorem \ref{thm:rt-interpretation}, we get that
    \[ \J_\alpha \in \frac{1}{(q)_\alpha^2} \ZZ[q^{\pm 1}]\]
    (since $A\cap \QQ(v^2) = \ZZ[v^{\pm2}]$). But it's easy to see from e.g. (\ref{eq:rewritten-recursion}) that $\J_\alpha$ has no pole at $q=0$, so in fact 
    \[ \J_\alpha \in \frac{1}{(q)_\alpha^2} \ZZ[q].\]
    
\end{proof}

\begin{remark} \label{rmk:palindromic}
    For $\alpha = \displaystyle\sum_{i=1}^n a_i \alpha_i \in Q_n^\pos$, it follows from Corollary \ref{cor:symmetry} (and the easy fact that $\J_\alpha|_{q=0}=1$) that the polynomial $(q)_\alpha^2 \J_\alpha$ is palindromic and monic of degree
    \[ 2\deg \left( (q)_\alpha \right) - \frac{(\alpha, \alpha)}{2}-(\rho,\alpha) = \sum_{i=1}^na_i(a_i+1) -  \frac{(\alpha, \alpha)}{2}-(\rho,\alpha).\]
    In the simply laced case, this simplifies to $\sum_{(i,j)\in E} a_ia_j$, where $E$ is the set of edges of the Dynkin diagram.

    Note that, in terms of the geometric interpretation of $\J_\alpha$, palindromicity follows from the fact that $Z^\alpha$ is Gorenstein \cite[Theorem 1.2]{BF} by a classical theorem of Stanley \cite[Theorem 4.1]{Sta}. However, it's not as clear how to obtain the degree from this perspective.
\end{remark}

\section{Combinatorial formula in type A} \label{sec:formula}

Let $Q_n = \{(x_j)_{j=1}^{n+1} \in \ZZ^{n+1} : \sum_{j=1}^{n+1} x_j = 0\}$ denote the root lattice of type $A_n$, which has a basis given by the simple roots $\alpha_1, \ldots, \alpha_n$, where \[\alpha_i=(0,\ldots, 0, 1, -1, 0, \ldots, 0)\] with the $1$ in position $i$. Let $Q_n^\pos$ denote the positive part of the root lattice, i.e. the $\NN$-span of $\alpha_1, \ldots, \alpha_n$. We reformulate Theorem \ref{thm:combi-formula} as follows.

\begin{definition} \label{def:A-numerator}
    For $\alpha = \displaystyle\sum_{i=1}^n a_i \alpha_i \in Q_n^\pos$, define polynomials $H_\alpha\in\ZZ[q]$ recursively by
    \begin{equation} \label{eq:A-recursion}
    H_\alpha = \sum_{\beta} \prod_{j=1}^{n-1} \qbinom{a_j}{b_j}_q \qbinom{a_{j+1}}{b_j}_q q^{\frac{(\beta, \beta)}{2}} H_\beta,
    \end{equation}
    where the sum is over all $\beta = \displaystyle\sum_{j=1}^{n-1} b_j \alpha_j \in Q_{n-1}^\pos$ if $n>0$, and $H_\alpha=1$ if $n=0$
\end{definition}

\begin{theorem} \label{thm:main-thm}
    For every $\alpha\in Q_n^\pos$, we have 
    \begin{equation}
        \J_\alpha = \frac{H_\alpha}{(q)_\alpha^2}
    \end{equation}
    where $H_\alpha$ is as in Definition \ref{def:A-numerator}.
\end{theorem}

It's clear that Theorem \ref{thm:main-thm} is an equivalent reformulation of Theorem \ref{thm:combi-formula}. Before proving it, we will need the following technical calculation:

\begin{lemma} \label{lem:technical-sum}
    Let $(a_i)_{i\in \ZZ}, (b_i)_{i\in \ZZ}$ be sequences of integers, finitely many of which are nonzero. Then we have the identity
    \begin{align*}
        &\sum_i \left(  (q^{a_{i+1}-a_i} -1) - q^{a_{i+1}-a_i}(1-q^{a_i-b_i})(1-q^{a_i-b_{i-1}}) \right) \\
        = &\sum_i \left((q^{b_{i+1}-b_i} -1) - q^{b_i-b_{i-1}}(1-q^{a_i-b_i})(1-q^{a_{i+1}-b_{i}})\right)
    \end{align*}
\end{lemma}

\begin{proof}
    The left-hand side expands to
    \begin{align*}
        \sum_i \left( q^{a_{i+1}-b_i} + q^{a_{i+1}-b_{i-1}} - q^{a_{i+1}+a_i-b_i - b_{i-1}} - 1 \right)
    \end{align*}
    and the right-hand side expands to
    \begin{align*}
        \sum_i \left( q^{b_{i+1}-b_i} - q^{b_{i}-b_{i-1}} + q^{a_{i}-b_{i-1}} + q^{a_{i+1}-b_{i-1}} - q^{a_{i+1}+a_i-b_i - b_{i-1}} - 1 \right) \\
        = \sum_i \left( q^{a_{i}-b_{i-1}} + q^{a_{i+1}-b_{i-1}} - q^{a_{i+1}+a_i-b_i - b_{i-1}} - 1 \right)
    \end{align*}
    by telescoping. The difference between the two sides is therefore
    \[ \sum_i \left(q^{a_{i+1}-b_i} - q^{a_{i}-b_{i-1}}\right) \]
    which vanishes by telescoping.
\end{proof}

\begin{proof}[Proof of Theorem \ref{thm:main-thm}]
    This is obvious for $\alpha=0$, so it's enough to check that $\frac{H_\alpha}{(q)_\alpha^2}$ satisfies the Toda recursion (\ref{eq:Toda-recursion}). Multiplying by $(q)^2_\alpha$ on both sides, we can rewrite this Toda recursion as
    \begin{equation} \label{eq:H-Toda-recursion}
        \left(\sum_{i} (q^{a_{i+1}-a_i} -1)\right) H_\alpha =  \sum_{i} q^{a_{i+1}-a_i} (1-q^{a_i})^2 H_{\alpha-\alpha_i},
    \end{equation}
    where for convenience we will have the sums run over all of $\ZZ$ with the convention that $a_i=0$ if $i\le 0$ or $i>n$. Note that if $a_i=0$ the coefficient of $H_{\alpha-\alpha_i}$ vanishes, so we don't have to worry about $H_{\alpha-\alpha_i}$ being undefined in this case.
    
    We prove (\ref{eq:H-Toda-recursion}) by induction on $n$, the base case $n=0$ being trivial. Note that, using (\ref{eq:A-recursion}) and the identity
    \[(1-q^a)\qbinom{a-1}{b}_q = (1-q^{a-b}) \qbinom{a}{b}_q,\]
    we have
    \begin{align*}
        (1-q^{a_i})^2H_{\alpha-\alpha_i}&= (1-q^{a_i})^2 \sum_{\beta} \qbinom{a_i-1}{b_i}_q \qbinom{a_{i}-1}{b_{i-1}}_q \prod_{j\ne i} \qbinom{a_j}{b_j}_q \qbinom{a_{j}}{b_{j-1}}_q q^{\frac{(\beta,\beta)}{2}} H_\beta \\
        &= (1-q^{a_i-b_i})(1-q^{a_i-b_{i-1}}) \sum_{\beta} \prod_{j} \qbinom{a_j}{b_j}_q \qbinom{a_{j}}{b_{j-1}}_q q^{\frac{(\beta,\beta)}{2}} H_\beta. 
    \end{align*}
    Substituting this into the right-hand side of (\ref{eq:H-Toda-recursion}), substituting (\ref{eq:A-recursion}) into the left-hand side, and bringing all the terms to the same side, we see that (\ref{eq:H-Toda-recursion}) is equivalent to the vanishing of
    \begin{align*}
        \sum_\beta \left(  (q^{a_{i+1}-a_i} -1) - q^{a_{i+1}-a_i}(1-q^{a_i-b_i})(1-q^{a_i-b_{i-1}}) \right) \\
        \prod_{j} \qbinom{a_j}{b_j}_q \qbinom{a_{j}}{b_{j-1}}_q q^{\frac{(\beta,\beta)}{2}} H_\beta.\qquad
    \end{align*}
    By Lemma \ref{lem:technical-sum}, this is equal to 
    \begin{align*}
        \sum_\beta\sum_i \left((q^{b_{i+1}-b_i} -1) - q^{b_i-b_{i-1}}(1-q^{a_i-b_i})(1-q^{a_{i+1}-b_{i}})\right) \\\prod_{j} \qbinom{a_j}{b_j}_q \qbinom{a_{j}}{b_{j-1}}_q q^{\frac{(\beta,\beta)}{2}} H_\beta.\qquad
    \end{align*}
    By the induction hypothesis, we have $\sum_i (q^{b_{i+1}-b_i} -1) H_\beta = \sum_{i} q^{b_{i+1}-b_i} (1-q^{b_i})^2 H_{\beta-\beta_i}$ where $\beta_i$ denote the $i^{th}$ simple root in $Q_{n-1}$. Substituting this in the expression above yields

    \begin{align*}
        \sum_\beta\sum_i \left(q^{b_{i+1}-b_i} (1-q^{b_i})^2 H_{\beta-\beta_i} - q^{b_i-b_{i-1}}(1-q^{a_i-b_i})(1-q^{a_{i+1}-b_{i}})H_\beta\right) \\\prod_{j} \qbinom{a_j}{b_j}_q \qbinom{a_{j}}{b_{j-1}}_q q^{\frac{(\beta,\beta)}{2}},\qquad
    \end{align*}
    which we can write as
    \begin{align*}
        \sum_i\sum_\beta \left(q^{b_{i+1}-b_i} (1-q^{b_i})^2 H_{\beta-\beta_i} - q^{b_i-b_{i-1}}(1-q^{a_i-b_i})(1-q^{a_{i+1}-b_{i}})H_\beta\right) \\
        \qbinom{a_i}{b_i}_q \qbinom{a_{i+1}}{b_{i}}_q q^{b_i^2-b_ib_{i-1} - b_ib_{i+1}} P_{i,\beta},\qquad
    \end{align*}
    where $P_{i,\beta}$ is a term that depends on $\beta$ only via the coefficients $b_j$ for $j\ne i$ (so $P_{i,\beta} =  P_{i,\beta+\beta_i}$). Using the identity
    \[(1-q^{a-b})\qbinom{a}{b}_q = (1-q^{b+1}) \qbinom{a}{b+1}_q,\]
    this can be rewritten again as
    \begin{align*}
    \sum_i\sum_\beta \biggl( &q^{b_{i+1}-b_i} (1-q^{b_i})^2 \qbinom{a_i}{b_i}_q \qbinom{a_{i+1}}{b_{i}}_q H_{\beta-\beta_i} \\
    - &q^{b_i-b_{i-1}}(1-q^{b_i+1})^2 \qbinom{a_i}{b_i+1}_q \qbinom{a_{i+1}}{b_{i}+1}_qH_\beta \biggr) q^{b_i^2-b_ib_{i-1} - b_ib_{i+1}} P_{i,\beta}. \\
    = \sum_i\sum_\beta \biggl( &q^{b_i(b_i-1)-b_ib_{i-1} - (b_i-1)b_{i+1}} (1-q^{b_i})^2 \qbinom{a_i}{b_i}_q \qbinom{a_{i+1}}{b_{i}}_q H_{\beta-\beta_i} \\
    - &q^{(b_i+1)b_i-(b_i+1)b_{i-1} - b_ib_{i+1}}(1-q^{b_i+1})^2 \qbinom{a_i}{b_i+1}_q \qbinom{a_{i+1}}{b_{i}+1}_qH_\beta \biggr) P_{i,\beta}. \\
    \end{align*}
    This last expression takes the form $\sum_i \sum_\beta (Q_{i,\beta}-Q_{i,\beta+\beta_i})$, for
    \[ Q_{i,\beta} = q^{b_i(b_i-1)-b_ib_{i-1} - (b_i-1)b_{i+1}} (1-q^{b_i})^2 \qbinom{a_i}{b_i}_q \qbinom{a_{i+1}}{b_{i}}_q H_{\beta-\beta_i} P_{i,\beta}.\]
    Since $Q_{i,\beta}=0$ when $b_i=0$, this vanishes by telescoping the sum over $b_i$.
\end{proof}

\begin{remark} \label{rmk:narayana}
    In the special case where $\alpha=\alpha_1+\ldots+\alpha_n$ is the highest root, triangular arrays $\underline{m}=(m_{k,i})_{1\le i\le  k \le n}$ satisfying $m_{n,i}=1$ and $m_{k,i} \le m_{k+1,i},m_{k+1,i+1}$ are naturally in bijection with Dyck paths of length $2n$. To a triangular array, we associate the path which separates the zeros from the ones in the array, for example:

    \begin{center}
    \begin{tikzpicture}
	\node at (1,0) [font=\fontsize{12}{0}\selectfont] {$1$};
    \node at (3,0) [font=\fontsize{12}{0}\selectfont] {$1$};
    \node at (5,0) [font=\fontsize{12}{0}\selectfont] {$1$};
    \node at (7,0) [font=\fontsize{12}{0}\selectfont] {$1$};
    \node at (2,1) [font=\fontsize{12}{0}\selectfont] {$0$};
    \node at (4,1) [font=\fontsize{12}{0}\selectfont] {$1$};
    \node at (6,1) [font=\fontsize{12}{0}\selectfont] {$1$};
    \node at (3,2) [font=\fontsize{12}{0}\selectfont] {$0$};
    \node at (5,2) [font=\fontsize{12}{0}\selectfont] {$0$};
    \node at (4,3) [font=\fontsize{12}{0}\selectfont] {$0$};
    \draw[thick,red] (0,0) -- (1,1) -- (2,0) -- (3,1) -- (4,2) -- (5,1) -- (6,2) -- (7,1) -- (8,0);
    \end{tikzpicture}
    \end{center}
    (here the array is drawn with the convention that $m_{k,j}$ is the $i^\text{th}$ number from the left on the $k^\text{th}$ row from top). 
    
    Moreover, under this bijection, $D(\underline{m})$ corresponds exactly to $n-1$ minus the number of valleys of the corresponding Dyck path (where a valley is a step down followed by a step up). To see this, note first that $D(\underline{m})$ is easily seen to count the number of $1$'s in the array $\underline{m}$ that are not in the bottom row and have no $1$'s directly to their right. One can check that every diagonal contains at most one such "rightmost $1$", and it contains one if and only if it has the same number of $0$'s as the diagonal directly below\footnote{Here we use \emph{diagonal} to mean a line consisting of entries $m_{k,i}$ for a fixed $i$, i.e. these are the diagonals with a bottom-left to top-right orientation.}. On the other hand, two consecutive diagonals have the same number of $0$'s if only if there is no valley between them. Since there are $n-1$ possible pairs of consecutive diagonals, this proves the claim. For example, for the array above, there are valleys between diagonals $1$ and $2$ and between diagonals $2$ and $3$, and there is a rightmost $1$ in diagonal $3$.

    It then follows from Theorem \ref{thm:combi-formula} that $K_\alpha$ is the $n^\text{th}$ Catalan number and that $(1-q)^{2n} \J_\alpha$ is the $n^\text{th}$ Narayana polynomial (which is the generating polynomial for the "number of valleys" statistic for Dyck paths). 
    
    Note that the $n^\text{th}$ Narayana polynomial is also known to arise as the numerator of the Hilbert series for the homogeneous coordinate ring of the Grassmannian of planes $G(2,n+2)$ in its Plucker embedding \cite{GW}. This raises the question of a possible connection (such as a flat degeneration) between this Grassmannian and the Zastava space $Z^\alpha$ which would explain this equality of Hilbert series. We do not know of such an explanation.

\end{remark}

\section{A conjectural topological interpretation} \label{sec:top-interpretation}

Empirically, the polynomial $(q)_\alpha^2\J_\alpha$ appears to always be unimodal in simply laced types. This naturally raises the question of whether there exists a smooth projective variety whose even Betti numbers are the coefficients of this polynomial (as this would explain unimodality by the hard Lefschetz theorem). In this section, we propose a candidate for such a variety in type A. Throughout this section, we let $\alpha = \displaystyle\sum_{i=1}^n a_i \alpha_i$ be an element of $Q_n^\pos$.

\begin{definition} \label{def:X-variety}
     Let $V_i$ be a vector space of dimension $a_i$ for $1\le i \le n$ and let $V=\displaystyle\bigoplus_{i=1}^n V_i$. Let $X_\alpha$ be the scheme over $\CC$ parameterizing triangular arrays $(W_{i,j})_{1\le i \le j \le n-1}$ of subspaces of $V$ satisfying
    \begin{itemize}
        \item $\dim W_{i,j}=\sum_{i\le k \le j} a_k$,
        \item $W_{i,j}\subseteq W_{i-1,j}$ if $i>1$ and  $W_{i,j}\subseteq W_{i,j+1}$ if $j<n-1$,
        \item $W_{i,n-1}\subseteq \displaystyle\bigoplus_{k=i}^n V_k$ and $W_{1,j}\subseteq \displaystyle\bigoplus_{k=1}^{j+1} V_k$.
    \end{itemize}
\end{definition}

For example, for $n=4$, $X_\alpha$ parametrizes fillings of the following diagram with subspaces $W_{i,j}$, where arrows represent inclusions and dimensions are indicated in red:

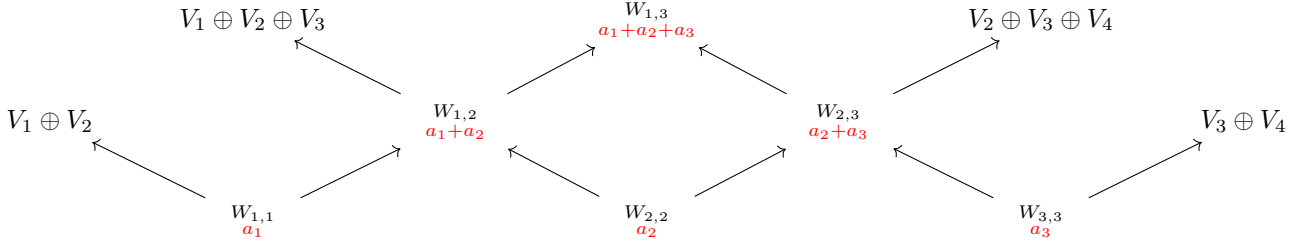
\[\adjustbox{scale=0.75,center}{%
\begin{tikzcd}
	& {V_1\oplus V_2 \oplus V_3} && \begin{array}{c} \substack{W_{1,3} \\ \color{red}a_1+a_2+a_3} \end{array} && {V_2\oplus V_3\oplus V_4} \\
	{V_1\oplus V_2} && \begin{array}{c} \substack{W_{1,2} \\ \color{red}a_1+a_2} \end{array} && \begin{array}{c} \substack{W_{2,3} \\ \color{red}a_2+a_3} \end{array} && {V_3\oplus V_4} \\
	& \begin{array}{c} \substack{W_{1,1} \\ \color{red}a_1} \end{array} && \begin{array}{c} \substack{W_{2,2} \\ \color{red}a_2} \end{array} && \begin{array}{c} \substack{W_{3,3} \\ \color{red}a_3} \end{array}
	\arrow[from=2-3, to=1-2]
	\arrow[from=2-3, to=1-4]
	\arrow[from=2-5, to=1-4]
	\arrow[from=2-5, to=1-6]
	\arrow[from=3-2, to=2-1]
	\arrow[from=3-2, to=2-3]
	\arrow[from=3-4, to=2-3]
	\arrow[from=3-4, to=2-5]
	\arrow[from=3-6, to=2-5]
	\arrow[from=3-6, to=2-7]
\end{tikzcd}}\]

Note that $X_\alpha$ is a closed subscheme of a product of Grassmannians, hence it is projective.

\begin{proposition} \label{prop:smooth}
   The scheme $X_\alpha$ is smooth, connected, and of dimension \[\sum_{i=0}^{n-1}a_ia_{i+1}.\]
\end{proposition}

\begin{proof}
    Let $\til{X}_\alpha$ parameterize triangular arrays $(W_{i,j})_{1\le i \le j \le n}$ of subspaces of $V$ satisfying
    \begin{itemize}
        \item $\dim W_{i,j}=\sum_{i\le k \le j} a_k$,
        \item $W_{i,j}\subseteq W_{i-1,j}$ if $i>1$ and  $W_{i,j}\subseteq W_{i,j+1}$ if $j<n$,
        \item $W_{1,j}\subseteq \displaystyle\bigoplus_{k=1}^{j+1} V_k$.
    \end{itemize}
    In other words, $\til{X}_\alpha$ is defined in the same way as $X_\alpha$ with the modification that the flag on the right of the diagram is now allowed to vary rather than being fixed to $(V_n, V_{n-1}\oplus V_n, \ldots, V_2\oplus \cdots \oplus V_n, V)$. Then $\til{X}_\alpha$ is smooth and connected, as it can be constructed as an iterated Grassmannian bundle starting from a point (by choosing the subspaces $W_{i,j}$ successively, where we run through all pairs $(i,j)$ with $1\le i,j\le n$, $(i,j)\ne (1,n)$ in lexicographic order). We can also see from this iterated Grassmannian bundle contruction that

    \[ \dim \til{X}_\alpha = \sum_{i=0}^{n-1}a_ia_{i+1} + \sum_{1 \le i < j \le n} a_{i}a_j, \]
    where the term indexed by $i$ in the first sum accounts for the choice of $W_{1,i}$ and the term indexed by $i,j$ in the second sum accounts for the choice of $W_{i+1, j}$.

    Moreover, there is a map $\pi:\til{X}_\alpha \to \Fl_\alpha$,  where $\Fl_\alpha$ is the partial flag variety of partial flags in $V$ of signature $(a_n, a_{n-1}+a_n, \ldots, a_1+\ldots+a_n)$, which remembers the right partial flag $(W_{n,n}, W_{n-1,n}, \ldots, W_{2,n}, W_{1,n})$. We have $X_\alpha = \pi^{-1}(F_\bullet^-)$, where $F_\bullet^-\in \Fl_\alpha$ is the partial flag $(V_n, V_{n-1}\oplus V_n, \ldots, V_2\oplus \cdots \oplus V_{n}, V)$.
    
    The map $\pi$ is equivariant for the action of the parabolic subgroup $P$ of $\GL(V)$ consisting of elements that preserve the partial flag $F_\bullet = (V_1, V_{1}\oplus V_2, \ldots, V_1\oplus \cdots \oplus V_{n-1}, V)$. Let $\Fl_\alpha^\circ = P\cdot F_\bullet^-\subset \Fl_\alpha$ be the open Bruhat cell and let $\til{X}^\circ_\alpha=\pi^{-1}(\Fl_\alpha^\circ)$. If $U$ is the unipotent radical of $P$, then the map $U \to \Fl_\alpha^\circ$ sending $u$ to $u\cdot F_\bullet^-$ is an isomorphism. It follows that the map $U\times X_\alpha \to \til{X}^\circ_\alpha$ sending $(u,p)$ to $(u\cdot F_\bullet^-, u\cdot p)$ is an isomorphism. Since $\til{X}^\circ_\alpha$ is smooth and connected, it follows that $X_\alpha$ is smooth and connected. Moreover,
    \begin{align*}
        \dim(X^\alpha)&= \dim(\til{X}_\alpha^\circ)-\dim(U) \\
        &= \left(\sum_{i=0}^{n-1}a_ia_{i+1} + \sum_{1 \le i < j \le n} a_{i}a_j\right)-\left( \sum_{1 \le i < j \le n} a_{i}a_j\right) \\
        &= \sum_{i=0}^{n-1}a_ia_{i+1}.
    \end{align*}
\end{proof}

The following conjecture is the main reason for our interest in the variety $X_\alpha$:

\begin{conjecture}\label{conj:poincare-poly}
    The Poincar\'e polynomial 
    \[ P_{X_\alpha}(t) = \sum_{i=0}^{2\dim X_\alpha} \dim H^{i}(X_\alpha)\cdot t^i\]
    of $X_\alpha$ is equal to $(q)_\alpha^2\J_\alpha|_{q=t^2}$.
\end{conjecture}

\begin{example}
    If $n=2$, and $\alpha =a_1\alpha_1+a_2\alpha_2$, then $X_{\alpha}$ is simply the Grassmannian $G(a_1,a_1+a_2)$, which has Poincar\'e polynomial $\qbinom{a_1+a_2}{a_1}_{t^2}$ as predicted by Conjecture \ref{conj:poincare-poly} (compare with Example \ref{ex:A2}).
\end{example}

\begin{example}
    In the setting of Remark \ref{rmk:narayana} (i.e. $\alpha=a_1+\ldots + \alpha_n$), $X_\alpha$ is a brick variety which is known to be isomorphic to the toric variety of the associahedron \cite[Corollary 2.10]{Esc}. By \cite[Theorem 12.3.12]{CLS}, the Betti numbers of this toric variety are given by the $h$-vector of the simplicial complex dual to the associahedron. The entries of this $h$-vector are exactly the Narayana numbers \cite[Lemma 5.7]{FR}, so Conjecture \ref{conj:poincare-poly} follows in that special case by our identification of $(q)_\alpha^2\J_\alpha$ as a Narayana polynomial.
\end{example}

As a sanity check for the reasonability of Conjecture \ref{conj:poincare-poly}, note that both $P_{X_\alpha}$ and $(q)_\alpha^2\J_\alpha|_{q=t^2}$ are monic palindromic polynomials of degree $2\sum_{i=1}^{n-1}a_ia_{i+1}$ (by Proposition \ref{prop:smooth} and Remark \ref{rmk:palindromic}).

Note also that $P_{X_\alpha}$ is really a polynomial in $t^2$, since $X_\alpha$ has no cohomology in odd degrees. This follows from the fact that $X_\alpha$ has a torus action (by a maximal torus in $\prod_{i=1}^n GL(V_i)$) with finitely many fixed points, from which we get a Bia\l{}ynicki-Birula decomposition of $X_\alpha$ into affine spaces. In fact, we can use this decomposition to give a combinatorial description of $P_{X_\alpha}$, reducing Conjecture \ref{conj:poincare-poly} to a combinatorial problem.

\begin{proposition}\label{prop:BB}
    For each $1\le i \le n$, let $S_i$ be totally ordered set with $a_i$ elements, and let $S=\bigsqcup_{i=1}^n S_i$, ordered so that elements of $S_i$ are smaller than elements of $S_j$ if $i<j$. Let $\mathcal{X}_\alpha$ be the set of triangular arrays $\ul{T}=(T_{i,j})_{1\le i\le j\le n-1}$ of subsets of $S$ satisfying
    \begin{itemize}
        \item $\# T_{i,j}=\sum_{i\le k \le j} a_k$,
        \item $T_{i,j}\subseteq T_{i-1,j}$ if $i>1$ and  $T_{i,j}\subseteq T_{i,j+1}$ if $j<n-1$,
        \item $T_{i,n-1}\subseteq \displaystyle\bigsqcup_{k=i}^n S_k$ and $T_{1,j}\subseteq \displaystyle\bigsqcup_{k=1}^{j+1} S_k$.
    \end{itemize}
    Given $\ul{T}\in \mathcal{X}_\alpha$, let 
    \[d_{\ul{T}}=\sum_{1\le i\le j \le n-1} d_{\ul{T}}^{i,j}-\sum_{1\le i<j\le n-1} a_i a_j,\]
    where $d_{\ul{T}}^{i,j}$ is the number of pairs $(s,t)$ with $s\in T_{i,j}\setminus T_{i,j-1}$, $t\in T_{i-1,j}\setminus T_{i,j}$ and $s > t$, with the convention that $T_{i,i-1}=\varnothing$ and $T_{0,j}=\bigsqcup_{i=1}^{j+1} S_i$. Then,
    \[P_{X_\alpha}(t) = \sum_{\ul{T}\in \mathcal{X}_\alpha} t^{2d_{\ul{T}}}.\]
\end{proposition}

\begin{proof}
    Pick a basis for each $V_i$ indexed by $S_i$ and let $T=(\CC^\times)^S \subseteq \prod_{i=1}^n \GL(V_i)$ be the corresponding maximal torus of diagonal matrices. Let $(\varepsilon_s)_{s\in S}$ be the standard basis for the weight lattice of $T$. Then it is clear that the $T$-fixed points of $X_\alpha$ are indexed by $\mathcal{X}_\alpha$: the fixed point corresponding to $\ul{T}$ is the array $(W_{i,j})$ where $W_{i,j}$ is the coordinate subspace spanned by the basis vectors indexed by elements of $T_{i,j}$. We denote by $p_{\ul{T}}$ this fixed point. If we choose a generic coweight $\chi: \CC^\times \to T$, we therefore obtain a Bia\l{}ynicki-Birula decomposition of $X_\alpha$ into affine spaces, from which it follows that
    \[P_{X_\alpha}(t) = \sum_{\ul{T}\in \mathcal{X}_\alpha} t^{2d'_{\ul{T}}},\]
    where $d'_{\ul{T}}$ is the dimension of the attracting set of $p_{\ul{T}}$, i.e. the number of weights $\lambda$ for the action of $T$ of $T_{p_{\ul{T}}}X_\alpha$ such that $\langle\lambda, \chi\rangle>0$, counted with multiplicities \cite{BB1, BB2}. It therefore remains to calculate $d_{\ul{T}}'$ and show that, for a suitable choice of coweight $\chi$, it equals $d_{\ul{T}}$.

    Let $\til{X}_\alpha$ be as in the proof of Proposition \ref{prop:smooth}. For every point \[p=(W_{i,j})_{1\le i\le j\le n} \in \til{X}_\alpha,\] by the construction of $\til{X}_\alpha$ as an iterated Grassannian bundle, we have a filtration of the tangent space $T_p\til{X}_\alpha$ whose subquotients are indexed by pairs $1\le i \le j \le n$, $(i,j)\ne (1,n)$. The subquotient indexed by $(i,j)$ is the tangent space at $\frac{W_{i,j}}{W_{i,j-1}}$ of the Grassmannian of $a_j$-dimensional subspaces of $\frac{W_{i-1,j}}{W_{i,j-1}}$, with the conventions that $W_{i,i-1}=\{0\}$ and $W_{0,j}=\bigoplus_{i=1}^{j+1} V_i$. Moreover, this tangent space can be canonically identified with $\Hom\left(\frac{W_{i,j}}{W_{i,j-1}}, \frac{W_{i-1,j}}{W_{i,j}}\right)$. We denote this particular subquotient by $T_p^{i,j}\til{X}_\alpha$. If $p=p_{\ul{T}}$ is a $T$-fixed point of $X_\alpha$, then this filtration is $T$-invariant and the isomorphism $T_{p_{\ul{T}}}^{i,j}\til{X}_\alpha\cong\Hom\left(\frac{W_{i,j}}{W_{i,j-1}}, \frac{W_{i-1,j}}{W_{i,j}}\right)$ is $T$-equivariant, so the weights of the $T$-action on $T_{p_{\ul{T}}}^{i,j}\til{X}_\alpha$ are given by $\e_t-\e_s$, as as we run through $s\in T_{i,j}\setminus T_{i,j-1}$ and $t\in T_{i-1,j}\setminus T_{i,j}$ (using the conventions $T_{i,i-1}=\varnothing$ and $T_{0,j}=\bigsqcup_{i=1}^{j+1} S_i$, as well as $T_{i,n}=\bigsqcup_{j=i}^n S_i$).

    If we choose $\chi$ so that $\langle \chi, \e_s\rangle > \langle \chi, \e_t\rangle$ whenever $s<t$, we obtain that the number of weights of $T_{p_{\ul{T}}}^{i,j}\til{X}_\alpha$ that pair positively with $\chi$ is given by $d_{\ul{T}}^{i,j}$, with $d_{\ul{T}}^{i,n}=\dim T_{p_{\ul{T}}}^{i,n}\til{X}_\alpha =a_{i-1}a_{n}$ as all weights of $T_{p_{\ul{T}}}^{i,n}\til{X}_\alpha$ pair positively with $\chi$. 
    
    Using the decomposition $\til{X}_\alpha^\circ \cong X_\alpha \times U$ from the proof of Proposition \ref{prop:smooth}, we have a $T$-equivariant decomposition $T_{p_{\ul{T}}} \til{X}_\alpha = T_{p_{\ul{T}}} X_\alpha \oplus \uf$, where $\uf$ is the Lie algebra of $U$. Note that all weights for the action of $T$ on $\uf$ pair positively with $\chi$ and $\dim \uf = \sum_{1\le i < j \le n} a_ia_j$, so we obtain from this decomposition that

    \begin{align*}
        d_{\ul{T}}' &= \sum_{1\le i \le j \le n-1} d_{\ul{T}}^{i,j} + \sum_{i=2}^n a_{i-1}a_n - \sum_{1\le i < j \le n} a_ia_j \\
        &= \sum_{1\le i \le j \le n-1} d_{\ul{T}}^{i,j} - \sum_{1\le i < j \le n-1} a_ia_j,
    \end{align*}
    which is our definition of $d_{\ul{T}}$.
    
\end{proof}

\begin{remark}
    The variety $X_\alpha$ can also be seen as a quiver Grassmannian. Specifically, let $Q$ be the quiver with vertices $v_{i,j}$ for $1 \le i \le j \le n-1$ and arrows $v_{i,j}\to v_{i,j+1}$ for $j<n-1$ and $v_{i,j}\to v_{i-1,j}$ for $i>1$. For example, for $n=5$:
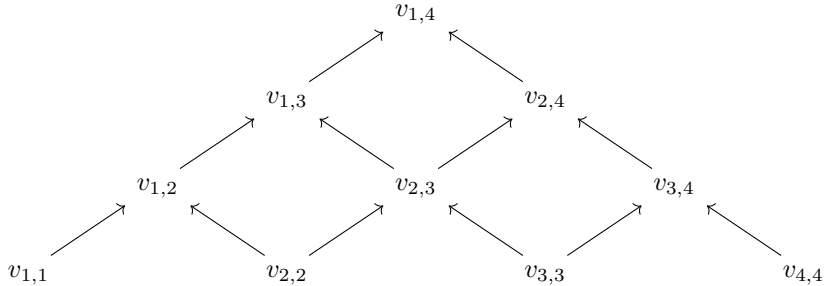
\[\begin{tikzcd}
	&&& {v_{1,4}} \\
	{} && {v_{1,3}} && {v_{2,4}} \\
	& {v_{1,2}} && {v_{2,3}} && {v_{3,4}} \\
	{v_{1,1}} && {v_{2,2}} && {v_{3,3}} && {v_{4,4}}
	\arrow[from=2-3, to=1-4]
	\arrow[from=2-5, to=1-4]
	\arrow[from=3-2, to=2-3]
	\arrow[from=3-4, to=2-3]
	\arrow[from=3-4, to=2-5]
	\arrow[from=3-6, to=2-5]
	\arrow[from=4-1, to=3-2]
	\arrow[from=4-3, to=3-2]
	\arrow[from=4-3, to=3-4]
	\arrow[from=4-5, to=3-4]
	\arrow[from=4-5, to=3-6]
	\arrow[from=4-7, to=3-6]
\end{tikzcd}\]

Let $M_\alpha$ be the representation of $Q$ whose vector space at vertex $v_{i,j}$ is $\bigoplus_{k=i}^{j+1} V_k$ (where $V_k$ is a fixed $a_k$-dimensional vector space as in Definition \ref{def:X-variety}) and all arrows are the obvious inclusions. Note that we have a direct sum decomposition

\[ M_\alpha = \bigoplus_{k=1}^n V_k \otimes M_{\alpha_k} \cong \bigoplus_{k=1}^n M_{\alpha_k}^{\oplus a_k}. \]
Then it is clear that $X_\alpha$ is the quiver Grassmannian parametrizing subrepresentations of $M_\alpha$ of dimension vector $d_{ij} = \sum_{k=1}^j a_k$. 

This perspective might be useful as Euler characteristics of quiver Grassmannians have been studied quite extensively before, and are known to be related to the theory of cluster algebras \cite{CC,DWZ,CK}. For any representation $M$ of $Q$, we consider the polynomial
\[F_M = \sum_{\ul{d}} \chi(\op{Gr}_{\ul{d}}(M)) \ul{x}^{\ul{d}}\] in $\binom{n}{2}$ variables $x_{ij}$, $1\le i\le j\le n-1$, where the sum is over all dimension vectors $\ul{d}=(d_{ij})_{1\le i\le j\le n-1}$, $\op{Gr}_{\ul{d}}(M)$ is the quiver Grassmannian parametrizing subrepresentations of $M$ of a given dimension vector $\ul{d}$, and $\ul{x}^{\ul{d}}=\prod_{i,j} x_{ij}^{d_{ij}}$. Then we have the multiplicativity property $F_{M\oplus N}=F_M F_N$ \cite[Proposition 3.6]{CC}, from which it follows that
\[ F_{M_\alpha} =\prod_{k=1}^n F_{M_{\alpha_k}}^{a_k}.\]
Moreover, $F_{M_{\alpha_k}}$ is very easy to compute as every nonempty quiver Grassmannian for $M_{\alpha_k}$ is a single point. We can therefore recover $\chi(X_\alpha)$ as a particular coefficient in this large product. Unravelling what this coefficient is essentially recovers the combinatorial interpretation of Proposition \ref{prop:BB} specialized at $t=1$, i.e. $\chi(X_\alpha)=\#\mathcal{X}_\alpha$.

\end{remark}

As further evidence for Conjecture \ref{conj:poincare-poly}, we conclude this section by proving that it holds in the first nontrivial case, which is in type $A_3$.

\begin{proposition}
    Conjecture \ref{conj:poincare-poly} holds for $n=3$.
\end{proposition}

\begin{proof}
    By Proposition \ref{prop:BB} specialized to $n=3$, we have
    \[ P_{X_\alpha}(t) = \sum_{\ul{T}\in \mathcal{X}_\alpha} (t^2)^{d_{\ul{T}}^{1,1}+d_{\ul{T}}^{1,2}+d_{\ul{T}}^{2,2}-a_1a_2},\]
    where the sum is over triples $(T_{1,1},T_{1,2},T_{2,2})$ with $\#T_{1,1}=a_1$, $\#T_{2,2}=a_2$, $\#T_{1,2}=a_1+a_2$ and $S_1 \sqcup S_2 \supseteq T_{1,1} \subseteq T_{1,2} \supseteq T_{2,2} \subseteq S_2\sqcup S_3$. We count the triples $(T_{1,1},T_{1,2},T_{2,2})$ according to $i=\#(T_{1,1}\cap S_2)$, $j=\#(T_{1,2}\cap S_3)$ and $k=\#(S_1\setminus T_{1,2})$, which yields
    
    \begin{align*}
        &\sum_{\ul{T}\in \mathcal{X}_\alpha} q^{d_{\ul{T}}^{1,1}+d_{\ul{T}}^{1,2}+d_{\ul{T}}^{2,2}-a_1a_2} \\
        &= \sum_{i,j,k} \left(\qbinom{a_1}{i}_q\qbinom{a_2}{i}_q q^{i^2} \right) \left(\qbinom{a_3}{j}_q \qbinom{i}{k}_q \qbinom{a_2-i}{j-k}_q q^{k(a_2-i-j+k)+kj+(j-k)j}\right) \\ &\hspace{40pt}\left(\qbinom{a_2+k}{a_2}_q q^{(a_1-k)a_2} \right)q^{-a_1a_2}\\
        &= \sum_{i,j,k} \qbinom{a_1}{i}_q\qbinom{a_2}{i}_q \qbinom{a_3}{j}_q \qbinom{i}{k}_q \qbinom{a_2-i}{j-k}_q \qbinom{a_2+k}{a_2}_q q^{i^2+j^2-ik-jk+k^2},
    \end{align*}
    where the sum is over $0\le i\le a_1, a_2$, $0\le j \le a_2, a_3$, and $0, i+j-a_2\le k \le i,j$. Here $\qbinom{a_1}{i}_q\qbinom{a_2}{i}_q q^{i^2}$ accounts for the sum of $q^{d_{\ul{T}}^{1,1}}$ over all possible choices of $T_{1,1}$, $\qbinom{a_3}{j}_q \qbinom{i}{k}_q \qbinom{a_2-i}{j-k}_q q^{k(a_2-i-j+k)+kj+(j-k)j}$ accounts for the sum of $q^{d_{\ul{T}}^{1,2}}$ over all possible choices of $T_{1,2}$ (with $T_{1,1}$ fixed) and $\qbinom{a_2+k}{a_2}_q q^{(a_1-k)a_2}$ accounts for the sum of $q^{d_{\ul{T}}^{2,2}}$ over all possible choices of $T_{2,2}$ (with $T_{1,1}$ and $T_{1,2}$ fixed). We can simplify the sum over $k$ by a sequence of applications of the $q$-Vandermonde identity. We isolate this calculation as Lemma \ref{lem:q-nanjundiah} below. Applying the lemma yields
    \begin{align*} \sum_{i,j,k} \qbinom{a_1}{i}_q\qbinom{a_2}{i}_q \qbinom{a_3}{j}_q \qbinom{i}{k}_q \qbinom{a_2-i}{j-k}_q \qbinom{a_2+k}{a_2}_q q^{i^2+j^2-ik-jk+k^2} \\
    = \sum_{i,j} \qbinom{a_1}{i}_q\qbinom{a_2}{i}_q \qbinom{a_2}{j}_q  \qbinom{a_3}{j}_q \qbinom{i+j}{j}_q q^{i^2+j^2-ij},
    \end{align*}
    which is exactly $(q)_\alpha^2\J_\alpha$ by Example \ref{ex:A3}.
\end{proof}

\begin{lemma} \label{lem:q-nanjundiah}
    For every integers $0\le i,j \le a$ we have the identity 
    \[\sum_{k} \qbinom{i}{k}_q \qbinom{a-i}{j-k}_q \qbinom{a+k}{a}_q q^{(i-k)(j-k)} = \qbinom{a}{j}_q \qbinom{i+j}{j}_q\]
\end{lemma}

\begin{proof}
    This identity can be seen as a $q$-analog of a special case of a binomial identity of Nanjundiah \cite[Equation (4)]{Nan}, and our proof is directly inspired by the proof in \cite{Nan}. We have
    \begin{align*}
        &\sum_{k} \qbinom{i}{k}_q \qbinom{a-i}{j-k}_q \qbinom{a+k}{a}_q q^{(i-k)(j-k)} \\
        &= \sum_{k,\ell} \qbinom{i}{k}_q \qbinom{a-i}{j-k}_q \qbinom{a}{\ell}_q \qbinom{k}{\ell}_q q^{\ell^2+(i-k)(j-k)} \\
        &= \sum_{k,\ell} \qbinom{i-\ell}{k-\ell}_q \qbinom{a-i}{j-k}_q \qbinom{a}{\ell}_q \qbinom{i}{\ell}_q q^{\ell^2+(i-k)(j-k)} \\
        &= \sum_{\ell} \qbinom{a-\ell}{j-\ell}_q  \qbinom{a}{\ell}_q \qbinom{i}{\ell}_q q^{\ell^2} \\
        &= \sum_{\ell} \qbinom{a}{j}_q  \qbinom{j}{\ell}_q \qbinom{i}{\ell}_q q^{\ell^2} \\
        &=\qbinom{a}{j}_q \qbinom{i+j}{j}_q, \\
    \end{align*}
    where we applied the $q$-Vandermonde identity \cite[Exercise 1.100]{Sta2} on the first, third and fifth lines.
\end{proof}

\section{Conclusion and open problems}

The proof of Theorem \ref{thm:denominator} relies on the representation-theoretic interpretation of $\J_\alpha$. It would be interesting to have a more elementary proof using the fermionic recursion directly (which we do not know how to do even in the $q\to 1$ limit of Remark \ref{rmk:K-limit}). It would also be very interesting to give a geometric explanation for the denominator, using the interpretation from Section \ref{sec:geo-interpretation}. Note that one factor of $(q)_\alpha$ in the denominator comes from the Hilbert series of $\bb{A}^\alpha$ (see Remark \ref{rmk:central-fiber}), but the other factor of $(q)_\alpha$ is not explained. The following conjecture would give a geometric proof of Theorem \ref{thm:denominator} if proven.

\begin{conjecture} \label{conj:regular-sequence}
    For every $\alpha\in Q^\pos$, there is a regular sequence \[(a_{i,j})_{1\le i \le n, 1\le j \le a_i}\] in $\CC[Z^\alpha_0]$ such that $a_{i,j}$ is homogeneous of degree $j$.
\end{conjecture}

Note that there does exist a regular sequence with the right number of elements (as Zastava spaces are known to be Cohen-Macaulay \cite[Theorem 1.2]{BF}), but the difficult part is getting the right degrees. Conjecture \ref{conj:regular-sequence} would in fact imply the following stronger form of Theorem \ref{thm:denominator}, asserting that the numerator has positive coefficients:

\begin{conjecture} \label{conj:pos-coeffs}
    For every $\alpha\in Q^\pos$, we have
    \[ \J_\alpha \in \frac{1}{(q)_\alpha^2}\NN[q].\]
\end{conjecture}

Conjecture  \ref{conj:pos-coeffs} holds in type A by Theorem \ref{thm:combi-formula}, but we do not know a proof in the general case. From empirical observations, we in fact believe that, in simply laced type, the polynomial $(q)_\alpha^2\J_\alpha$ not only has positive coefficients, but is a unimodal polynomial.

\begin{conjecture} \label{conj:unimodal}
    For simply laced root systems\footnote{A previous version of this paper incorrectly stated this conjecture for all types, but it in fact already fails in type $\text{B}_2$ for $\alpha=2\alpha_1+2\alpha_2$, in which case the numerator $(q)_\alpha^2\J_\alpha$ is given by $1+q+3q^2+2q^3+5q^4+2q^5+3q^6+q^7+q^8$.}, the polynomial $(q)_\alpha^2\J_\alpha$ is unimodal.
\end{conjecture}

Conjecture \ref{conj:unimodal} would follow from Conjecture \ref{conj:poincare-poly} in type A by the Hard Lefschetz theorem. Outside type A, unimodality might also be explained by the existence of a smooth variety with $(q)_\alpha^2\J_\alpha$ as its Poincaré polynomial, though this would require further investigations as we do not currently have a candidate for such a variety in types D and E.

We would also like to see a more conceptual proof of the formula from Theorem \ref{thm:combi-formula}, using either the geometric or representation-theoretic interpretation of $\J_\alpha$. For example, if one could find a regular sequence $(a_{i,j})$ as in Conjecture \ref{conj:regular-sequence} and a nice basis for the finite-dimensional graded algebra $\CC[Z^\alpha_0]/(a_{i,j})$ of a combinatorial flavor, it could yield a new proof of Theorem \ref{thm:combi-formula}. One reason why a more conceptual proof is desirable is that it might lead to combinatorial formulas for $(q)_\alpha^2\J_\alpha$ similar to (\ref{eq:main-formula}) in types other than type A.

Finally, let us mention that, assuming Conjecture \ref{conj:regular-sequence}, both sides of Conjecture \ref{conj:poincare-poly} are the generating polynomials for the graded dimension of some finite-dimensional graded algebra: the first is $H^\bullet(X_\alpha)$ and the second is the quotient $\CC[Z^\alpha_0]/(a_{i,j})$. One might therefore hope that, at least for a suitable choice of regular sequence $(a_{i,j})$, Conjecture \ref{conj:poincare-poly} is the combinatorial shadow of an isomorphism of graded algebras $H^\bullet(X_\alpha)\cong\CC[Z^\alpha_0]/(a_{i,j})$. This hypothetical isomorphism is quite reminiscent of Hikita's conjecture, which predicts an isomorphism between the cohomology ring of a symplectic resolution and a certain finite dimensional quotient of the coordinate ring of its symplectic dual \cite{Hik} \cite[Section 5.6]{Kam}. While our study case does not fit directly in Hikita's setting (as Zastava spaces do not have symplectic resolutions), the similarity is striking enough to wonder whether there is a hidden relationship.

\section{Acknowledgements}

The author would like to thank his advisor Joel Kamnitzer for his guidance and comments on the paper, as well as Dave Anderson, Leonid Rybnikov, Vasily Krylov, Alexander Braverman, Evgeny Feigin and Gabe Udell for interesting discussions. We would also like to thank Allen Knutson for suggesting the idea behind the proof of Proposition \ref{prop:smooth}. Finally, we are thankful to Haobo Ma and Wenlin Zhang for pointing out the counterexample to Conjecture \ref{conj:unimodal} in type $\text{B}_2$.

\bibliography{main.bib}
\bibliographystyle{plain}

\Address

\end{document}